\documentclass[11pt]{amsart}
\usepackage{amssymb,amsmath,amsthm,amsfonts,amsopn,url}
\usepackage{enumitem}
\theoremstyle{plain}
\newtheorem{thm}{Theorem}[section]

\newtheorem{prop}[thm]{Proposition}
\newtheorem{lemma}[thm]{Lemma}
\newtheorem{cor}[thm]{Corollary}

\theoremstyle{definition}
\newtheorem{defn}[thm]{Definition}
\newtheorem*{defn*}{Definition}
\newtheorem{example}[thm]{Example}
\newtheorem*{example*}{Example}

\theoremstyle{remark}

\newcommand{\field}[1]{\mathbb{#1}}

\newcommand{\N}{\field{N}}

\newcommand{\Z}{\field{Z}}
\newcommand{\ideal}[1]{\mathfrak{#1}}
\newcommand{\m}{\ideal{m}}

\DeclareMathOperator{\pdeg}{pdeg}

\renewcommand{\phi}{\varphi}

\author{Neil Epstein}
\address{Department of Mathematical Sciences \\ George Mason University \\ Fairfax, VA  22030}
\email{nepstei2@gmu.edu}

\author{Jay Shapiro}
\address{Department of Mathematical Sciences \\ George Mason University \\ Fairfax, VA  22030}
\email{jshapiro@gmu.edu}

\title{A Dedekind-Mertens theorem for power series rings}


\date{October 8, 2014}
\begin{document}
\begin{abstract}
We prove a power series ring analogue of the Dedekind-Mertens lemma.  Along the way, we give limiting counterexamples, we note an application to integrality, and we correct an error in the literature.
\end{abstract}

\maketitle

\section{Introduction}
Let $R$ be a commutative ring, and let $S=R[X]$.  For $f\in S$, the \emph{content} of $f$, written $c(f)$, is the ideal of $R$ generated by the coefficients of $f$.

One form of a lemma of Gauss states that if $R=\Z$ and $S=\Z[X]$, then for any $f, g \in S$, one has $c(fg) = c(f)c(g)$.  This result and its avatars are useful for instance in extending the property of unique factorization to polynomial rings.  However, in this form, Gauss's lemma is not true for general commutative rings $R$.  Indeed, if $R$ is an integral domain, it's only true if $R$ is a Pr\"ufer domain (which in the Noetherian case means that $R$ is a Dedekind domain).  For this and more on Gauss's lemma in rings with zero-divisors, see the survey article~\cite{GlSc-Pruefer}.

The Dedekind-Mertens lemma (proved by Dedekind \cite{Ded-DM} and independently in a weaker form by Mertens \cite{Mer-DM}, both in 1892) is a generalization of Gauss's lemma that works for all commutative rings.  Namely, they show that $c(f)^k c(g) = c(f)^{k-1} c(fg)$ for some $k \in \N$ (where in Dedekind's case, $k=1+\deg g$).  Of course, the general notion of ring did not exist in the 1890's, so Dedekind and Mertens were assuming $R$ to be the ring of integers of a number field, but Dedekind's proof works over any commutative ring, as verified by Pr\"ufer in 1932 \cite[p.24]{Pr-DM}.  See \cite[footnote 1]{HeiHu-content} for a brief discussion.  Heinzer and Huneke \cite[main theorem]{HeiHu-content} show that one may improve the Dedekind-Mertens lemma by letting $k$ be the locally minimal number of generators of $c(g)$.  For more history and context regarding the many forms of the Dedekind-Mertens Lemma, see \cite[Section 8]{An-gcdsurvey}.

It is natural to ask what happens when one replaces the polynomial ring over $R$ with a power series ring in one variable.  Is there an integer $k$ that works in this case?  So let $S=R[\![X]\!]$ for the remainder of this discussion.  There too, one may define the content of a power series to be the ideal generated by the coefficients.

Gilmer, Grams and Parker \cite[Theorem 3.6]{GGP-power} showed that in this case (and more generally, for greater sets of power series variables) if $g$ is a polynomial of degree $k-1$, then even when $f$ is a power series, one has $c(f)^k c(g) = c(f)^{k-1} c(fg)$.  On the other hand, Rush \cite{Ru-content} claims to exhibit a counterexample when $f$, $g$ are both power series of infinite degree, even when $R$ is a two-dimensional polynomial ring over a field.  Rush's counterexample is wrong, however.  Indeed, in our main theorem, we show that the desired generalization holds whenever $R$ is a Noetherian ring.  Moreover, the exponent $k$ that we use is the locally minimal number of generators of $c(g)$, as in the Heinzer-Huneke version.

In the last section, we show what is wrong with Rush's counterexample and we give an example where $R$ is non-Noetherian to show that at that level of generality, there may be no $k$ such that $c(f)^k c(g) = c(f)^{k-1} c(fg)$.

In the proof of our main theorem, we follow very closely the structure of the proof of the main theorem from \cite{HeiHu-content}.

\section{Results}

The first result we will need was proved in Hwa Tsang's Ph.D. thesis \cite[Chapter II, Lemma 1.2]{Ts-Gauss} in a more general form.  Namely, she proved the result without assuming $R$ is Noetherian, but only that $c(f)=R$ and that the ideal $c(g)$ is finitely generated.  However, our proof is completely different and may be of independent interest, so we include it below.

\begin{prop}\label{pr:pssemi}
Let $R$ be a Noetherian ring, and let $f, g \in R[\![X]\!]$ such that $c(f) = R$.  Then $c(fg) = c(g)$.
\end{prop}

\begin{proof}
It is clear that $c(fg) \subseteq c(f)c(g) = c(g)$.  Since $R$ is Noetherian, we may take a primary decomposition of $c(fg)$, namely: \[
c(fg) = Q_1 \cap \cdots \cap Q_k,
\]
where each $Q_i$ is  primary.  For each $i$, we have $fg \in Q_i R[\![X]\!]$, but $f \notin \sqrt{Q_i R[\![X]\!]}$ since $c(f)=R$.  However, it is shown in \cite[Theorem 8]{Br-psbook} that $Q_i R[\![X]\!]$ is primary in this context.  Hence, it follows that $g \in Q_i R[\![X]\!]$, whence $c(g) \subseteq Q_i$.  Since this holds for all $1\leq i \leq k$, it follows that $c(g) \subseteq c(fg)$, as was to be shown.
\end{proof}

\begin{lemma}\label{lem:pscoeff}
Let $R$ be a Noetherian ring, and let $f \in R[\![X]\!]$.  It is possible to write $f$ in the form \[
f = \sum_{j=0}^n a_j u_j X^j,
\]
where $a_j \in R$ and $u_j$ is a unit of $R[\![X]\!]$ for each $0 \leq j \leq n$.  For any such representation of $f$, we have $c(f) = (a_0, \dotsc, a_n)$.
\end{lemma}

\begin{proof}
For the first part, write $f = \sum_{i=0}^\infty a_i X^i$, with $a_i \in R$.  The ideals $(a_0) \subseteq (a_0, a_1) \subseteq \cdots$ form an ascending chain, so since $R$ is Noetherian, there must be some $n$ such that for all $i> n$, $a_i \in (a_0, a_1, \dotsc, a_n)$.  In particular, for any such $i$, we may write \[
a_i = \sum_{j=0}^n r_{ij} a_j.
\]
Then we have \begin{align*}
f &= \left(\sum_{j=0}^n a_j X^j\right) + \sum_{i=n+1}^\infty a_i X^i\\
&= \left(\sum_{j=0}^n a_j X^j\right) + \sum_{i=n+1}^\infty \left(\sum_{j=0}^n r_{ij} a_j\right) X^i \\
&= \sum_{j=0}^n a_j \left(1 + \sum_{i=n+1}^\infty r_{ij}X^{i-j}\right) X^j.
\end{align*}
Setting $u_j := 1 + \sum_{i=n+1}^\infty r_{ij} X^{i-j}$ for each $0\leq j \leq n$, we have represented $f$ in the required form.

For any such representation, clearly $c(f) \subseteq (a_0, \dotsc, a_n)$ by subadditivity of the content function.  For the reverse inclusion, we show by induction on $j$, where $0 \leq j \leq n$, that $a_0, \ldots, a_j \in c(f)$.  We use the vacuous case as the base of the induction.  
So let $0 \leq j \leq n$, and suppose we have shown that $a_0, \dotsc, a_{j-1} \in c(f)$.  Write $u_i = \sum_{k=0}^\infty u_{i,k} X^k$, $u_{i,k} \in R$.  Then the $R$-coefficient of $X^j$ is \[
\sum_{i=0}^j a_i u_{i, j-i} \in c(f).
\]
But since each $a_i$ is in $c(f)$ for $i<j$, it follows that $a_j u_{j,0} \in c(f)$.  But $u_{j,0}$ is a unit, so $a_j \in c(f)$.
\end{proof}

Inspired by the above lemma, we make the following definition:

\begin{defn}
Let $f \in R[\![X]\!]$.  Then the \emph{pseudodegree} of $f$, written $\pdeg f$, is the smallest nonnegative integer $k$ such that there exists a representation of $f$ in the form \[
f=\sum_{j=0}^k a_j u_j X^j,
\]
where each $a_j \in R$ and each $u_j$ is a unit of $R[\![X]\!]$.
\end{defn}

Next we prove a lemma that is comparable to \cite[Lemma 2.4]{HeiHu-content}.

\begin{lemma}\label{lem:mingen}
Let $(R,\m)$ be a Noetherian local ring.  Let $f, g \in R[\![X]\!]$, $i \in \N$, $b \in \m c(g)$, and $h := g+buX^i$, where  $u \in R[\![X]\!]$ is a unit.  Then $c(g)=c(h)$, and for any ideal $J$ of $R$ such that $J c(f)c(h) = J c(fh)$, one has $J c(f) c(g) = J c(fg)$.
\end{lemma}

\begin{proof}
For the first part, first note that by subadditivity, we have $c(h) \subseteq c(g) + c(buX^i) \subseteq c(g) + bR = c(g)$ (since $b \in c(g)$).  Then we have $c(g) \subseteq c(h) + c(buX^i) =$ (by Proposition~\ref{pr:pssemi}) $c(h) + c(b) \subseteq c(h) + \m c(g)$.  Finally, the desired the equality follows from the Nakayama lemma applied to the quotient module $c(g) / c(h)$.

For the second part, we have \begin{align*}
J c(f)c(g) &= J c(f) c(h) = J c(fh) = J c(fg + buX^if) \\
&\subseteq J c(fg) + \m J c(f) c(g),
\end{align*}
and then the result follows from the Nakayama lemma applied to the quotient module $\frac{J c(f)c(g)}{J c(fg)}$.
\end{proof}

The following lemma may be well known, but we could not find a reference to it, so we prove it here for the convenience of the reader. Here $\mu(I)$ denotes the minimal number of generators of an ideal $I$.

\begin{lemma}\label{lem:ginduction}
Let $(R,\m)$ be a local ring.  Let $I$ be a finitely generated ideal with $\mu(I) = k\geq 2$.  Say $I = (c_1, \dotsc, c_k)$, and let $J$ be a finitely generated ideal with $J \subseteq (c_1, \dotsc, c_{k-1})$ and $J + (c_k) = I$.  Then $J = (c_1, \dotsc, c_{k-1})$.
\end{lemma}

\begin{proof}
For each $1\leq i \leq {k-1}$, there exist $a_i \in R$ and $x_i \in J$ such that \[
c_i = a_i c_k + x_i.
\]
On the other hand, since $J \subseteq (c_1, \dotsc, c_{k-1})$, we may write \[
x_i = \sum_{j=1}^{k-1} r_{ij} c_j \in J.
\]
Thus, we have \[
0 = a_i c_k + (r_{ii} - 1) c_i + \sum_{\stackrel{2\leq j \leq k}{j \neq i}} r_{ij} c_j.
\]
Since $c_1, \dotsc, c_k$ are a minimal generating set for $I$, it follows that each of the coefficients of the $c_j$s in the above equation are in $\m$.  That is, $a_i \in \m$, $r_{ij} \in \m$ whenever $i \neq j$, and $r_{ii}$ is a \emph{unit}.

Let $M$ be the $(k-1) \times (k-1)$ matrix whose entry in the $(i,j)$ spot is $r_{ij}$.  Then all the entries of the vector \[
M \cdot \left(\begin{matrix} c_1 \\ \vdots \\ c_{k-1}\end{matrix} \right)
\]
are in $J$.  It follows from Cramer's rule that for each $1\leq i<k$, $\det(M)c_i \in J$.  But the classical expansion of $\det(M)$ looks like \[
r_{11} r_{22} \cdots r_{k-1,k-1} + \text{(other terms)},
\]
where each summand in the ``other terms'' part is a multiple of some $r_{ij}$ with $i\neq j$.  Hence, the other terms are all in $\m$.  But the first summand is a product of units of $R$, hence a unit, so that $\det(M)$ is itself a unit of $R$.  Thus, $c_i \in J$ for each $1\leq i\leq k-1$, finishing the proof.
\end{proof}

\begin{thm}\label{mainthm}
Let $R$ be a Noetherian ring, and let $0 \neq g\in R[\![X]\!]$.  Let $k$ be the maximum of the numbers $\mu(c(g)_\m)$, taken over all maximal ideals $\m$ of $R$.  (In particular, $\mu(c(g)) \geq k$.)  Then for all $f \in R[\![X]\!]$, we have \begin{equation}\label{eqn:ceq}
c(f)^k c(g) = c(f)^{k-1} c(fg).
\end{equation}
\end{thm}

\begin{proof}
First, note that we may immediately reduce to the case where $(R,\m)$ is local, and moreover in this case we may assume that $g\neq 0$ and $k=\mu(c(g))$.

The proof proceeds by induction on $k$ for Equation~\ref{eqn:ceq}.  If $k=1$, then $c(g) = (a)$ for some $0 \neq a\in R$, whence $g = ah$ for some $h \in R[\![X]\!]$.  Then $(a)=c(g) = c(ah) = ac(h)$, so for some $t\in c(h)$, we have $a(1-t)=0$.  Since $a\neq 0$, it follows that $1-t$ cannot be a unit, whence $t\notin \m$, so that $c(h) = R$.  Then $c(fg) = c(afh) = ac(fh) = ac(f) $ (by Proposition~\ref{pr:pssemi}) $= c(f)c(g)$.  So from now on we fix a $g$ with $\mu(c(g)) =k\geq 2$, and we may assume that we have proved the result for all $f$ when $k$ is smaller.

Now, write $g = X^s \cdot \sum_{i=0}^m b_i u_i X^i$, where $b_0 \neq 0$, each $b_i \in R$, and each $u_i$ is a unit of $R[\![X]\!]$.  By repeated use of Lemma~\ref{lem:mingen}, we may assume that $b_0 \notin \m c(g)$.  Next, we may divide by $X^s$ and assume that $s=0$.  On the other hand, Lemma~\ref{lem:pscoeff} gives that $c(g) = (b_0, \dotsc, b_m)$.  Thus, there is some minimal system of generators $c_1, \dotsc, c_k$ of $c(g)$ such that $c_k = b_0$.  For each $1\leq i \leq m$, we may write $b_i = \sum_{h=1}^k \lambda_{ih} c_h$, where each $\lambda_{ih} \in R$.  Collecting coefficients, we have \[
g = b_0 u + g_1,
\]
where $u= u_0 + \sum_{i=1}^m \lambda_{ik}  u_i X^i$ and $g_1 = \sum_{i=1}^m (\sum_{h=1}^{k-1} \lambda_{ih} c_h) u_i X^i$.  Observe that $u$ is a unit of $R[\![X]\!]$.  Also, $c(g_1) \subseteq (c_1, \dotsc, c_{k-1})$ and we have \[
c(g) \subseteq c(b_0 u) + c(g_1) = (c_k)+ c(g_1) \subseteq (c_k) + (c_1, \dotsc, c_{k-1}) = c(g),
\]
whence $c(g_1) + (c_k) = c(g)$.  At this point, Lemma~\ref{lem:ginduction} applies, so that $c(g_1) = (c_1, \dotsc, c_{k-1})$.  In particular, $\mu(c(g_1)) \leq k-1$.

Now we begin an induction on $n=\pdeg f$ for Equation~\ref{eqn:ceq} (keeping in mind that our $g$ is fixed until the end of the proof).  If $n=0$, then $f=au$ for some $a\in R$ and some unit $u$ of $R[\![X]\!]$, and we have $c(fg) = c(au g) = ac(ug) = ac(g)$ (by Proposition~\ref{pr:pssemi}) $= c(f) c(g)$, so that $c(f)^k c(g) = c(f)^{k-1} c(f)c(g) = c(f)^{k-1} c(fg)$.  So now we fix an $f$ with $\pdeg f = n \geq 1$, and assume inductively that we have proved the equation for all $f$s with smaller pseudodegrees (for our fixed $g$).

In particular, write $f = \sum_{i=0}^n a_i v_i X^i$ ($a_i \in R$, $v_i$ units of $R[\![X]\!]$) and set $f_1 := (f-a_0 v_0)/X= \sum_{i=0}^{n-1} a_{i+1} v_{i+1} X^i$, where each $a_i \in R$ and each $v_i$ is a unit of $R[\![X]\!]$.  Note that $a_0 b_0 \in c(fg)$.

Since the pseudodegree of $f_1$ is at most $n-1$, by the inductive hypothesis on the pseudodegree of $f$, we have \[
c(f_1)^k c(g) = c(f_1)^{k-1} c(f_1g).
\]

\noindent \textit{Claim 1}: $c(fg_1) \subseteq c(fg) + b_0 c(f_1)$.

\begin{proof}[Proof of claim]
\begin{align*}
c(fg_1) &= c(fg-b_0 fu) \subseteq c(fg) + b_0 c(fu) \subseteq c(fg) + b_0 c(f) \\
&= c(fg) + b_0 c(a_0 v_0 + Xf_1) \subseteq c(fg) + a_0 b_0 R + b_0 c(Xf_1) \\
&= c(fg) + b_0 c(f_1).
\end{align*}
\end{proof}

\noindent \textit{Claim 2:} $c(f_1g) \subseteq c(fg) + a_0 c(g_1)$.

\begin{proof}[Proof of claim]
\begin{align*}
c(f_1g) &= c(((f-a_0v_0)/X)g) = c((f-a_0v_0)g) \subseteq c(fg) + a_0c(g) \\
&\subseteq c(fg) + a_0 c(b_0u + g_1) \subseteq c(fg) + a_0 b_0 R + a_0 c(g_1)= c(fg) + a_0 c(g_1).
\end{align*}
\end{proof}

By Lemma~\ref{lem:pscoeff}, the ideal $c(f)^{k} c(g)$ is generated by all terms of the form \[
\omega = a_0^{\ell_0} a_1^{\ell_1} \cdots a_n^{\ell_n} b_j,
\]
where $\sum_{i=0}^n \ell_i = k$ and $0 \leq j \leq m$.  So to prove Equation (\ref{eqn:ceq}), it suffices to show that any such term is an element of $c(f)^{k-1} c(fg)$.  

\vspace{3pt}
\noindent \textit{Case 1}: Suppose $\ell_0\neq 0$ and $j=0$.   Then \[
\alpha = (a_0^{\ell_0-1} a_1^{\ell_1}\cdots a_n^{\ell_n}) \cdot a_0 b_0 \in c(f)^{k-1} \cdot c(fg).
\]

\noindent \textit{Case 2:} Suppose $\ell_0 \neq 0$ and $j>0$.  Then \[
\alpha = (a_0^{\ell_0-1} a_1^{\ell_1} \cdots a_n^{\ell_n}) \cdot a_0 b_j \in c(f)^{k-1} a_0 c(g_1).
\]

\noindent \textit{Case 3:} Suppose $\ell_0=0$.  Then $\alpha \in c(f_1)^{k} c(g) = c(f_1)^{k-1} c(f_1 g)$ by induction on the pseudodegree of $f$.
\vspace{3pt}

Combining these cases together, we have \begin{align*}
c(f)^{k}c(g) &\subseteq c(f)^{k-1} c(fg) + c(f)^{k-1} a_0 c(g_1) + c(f_1)^{k-1} c(f_1 g) \\
&\subseteq c(f)^{k-1} c(fg) + c(f)^{k-1} a_0 c(g_1) + c(f_1)^{k-1} (c(fg) + a_0 c(g_1)) \\
&\subseteq c(f)^{k-1} c(fg) + a_0 c(f)^{k-1} c(g_1),
\end{align*}
where the second containment is by Claim 2 and the third follows from the fact that $c(f_1) \subseteq c(f)$ (by Lemma~\ref{lem:pscoeff}).  On the other hand, since $\mu(c(g_1)) \leq k-1$, the induction hypothesis on $k$ gives that $c(f)^{k-1} c(g_1) = c(f)^{k-2} c(fg_1)$.  Hence, the sequence of containments and equalities above continues as follows: \begin{align*}
\cdots &= c(f)^{k-1} c(fg) + a_0 c(f)^{k-2} c(fg_1) \\
&\subseteq c(f)^{k-1} c(fg) + a_0 c(f)^{k-2} (c(fg) + b_0 c(f_1)) \\
&\subseteq c(f)^{k-1} c(fg),
\end{align*}
where the first containment follows from Claim 1, and the second follows from the containments $a_0 \in c(f)$, $a_0 b_0 \in c(fg)$, and $c(f_1) \subseteq c(f)$.
\end{proof}

Recall that for a pair of ideals $J \subseteq I$, we say that $J$ is a \emph{reduction} of $I$ if there is some $r\in \N$ such that $I^{r+1} = JI^r$.  Hence, our theorem implies the following:

\begin{cor}
Let $R$ be a Noetherian ring, and $f, g\in R[\![X]\!]$.  Then $c(fg)$ is a reduction of $c(f)c(g)$, with reduction number at most $k-1$, where $k$ is as in the theorem.
\end{cor}

\begin{proof}
It is elementary that $c(fg) \subseteq c(f)c(g)$.  On the other hand, multiply Equation~\ref{eqn:ceq} by the quantity $c(g)^{k-1}$, and we get \[
(c(f)c(g))^k = c(f)^k c(g) c(g)^{k-1} = c(f)^{k-1}c(fg) c(g)^{k-1} = c(fg) \cdot(c(f)c(g))^{k-1}.
\]
\end{proof}

\section{A generalization to modules}
All of the above can be done with modules.  That is, let $R$ be a Noetherian ring, let $M$ be a finitely generated module, and let $g\in M[\![X]\!]$.  Then $c(g)$, the content of $g$, is defined to be the \emph{submodule} of $M$ generated by its coefficients.  As $M$ is a Noetherian module, $c(g)$ will be a finitely generated submodule.  We have the following theorem 

\begin{thm}
Let $R$ be a Noetherian ring, $M$ a finitely generated module, $f\in R[\![X]\!]$, and $g \in M[\![X]\!]$.  Let $k$ be the maximum among the numbers $\mu(c(g)_\m)$, where $\m$ ranges over all maximal ideals of $R$.  Then \[
c(f)^k c(g) = c(f)^{k-1}c(fg).
\]
\end{thm}

The proof is then \emph{exactly} like the proof of Theorem~\ref{mainthm}.  Of course, in order to prove the analogue of the crucial Proposition~\ref{pr:pssemi}, one needs an analogue of \cite[Theorem 8]{Br-psbook}.  But in fact, one may just copy Brewer's proof, using the standard translation between primary ideals and primary submodules.

\section{Examples and a question}

For general commutative rings, there is no exponent for which the content formula holds in power series rings (at least for the notion of content used here), as shown in the following counterexample.
\begin{example}
Let $F$ be a field, let $a_0, a_1, \ldots$ and $b_0, b_1, \ldots$ be two sequences of indeterminates over $F$.  Let $R = F[a_0, a_1, \ldots, b_0, b_1, \ldots]$, let $S = R[\![X]\!]$, and consider the power series $f=\sum_{i=0}^\infty a_i X^i$ and $g=\sum_{i=0}^\infty b_i X^i$ in $S$.

Suppose that there is some $k\in \N$ such that \[
c(f)^k c(g) = c(f)^{k-1}c(fg).
\]
Let $R' := F[a_0, \ldots, a_k, b_0, \ldots, b_k]$, and define $\pi: R \twoheadrightarrow R'$ by sending each $a_i$ (resp $b_i$) to itself whenever $i \leq k$, and to $0$ otherwise.  Let $f' := \pi(f) = \sum_{i=0}^k a_i X^i \in R'[X]$ and $g' := \pi(g) = \sum_{i=0}^k b_i X^i$.  Note that $\pi(c(f)) = c(f')$, $\pi(c(g)) = c(g')$, and $\pi(c(fg)) = c(f'g')$.

According to \cite[Theorem 2.1]{CVV-Gauss}, we have that \[
(c(f')c(g'))^k \neq (c(f')c(g'))^{k-1} c(f'g').
\]
It follows that $c(f')^k c(g') \neq c(f')^{k-1} c(f'g')$.  But then an application of $\pi$ to the first displayed equation yields a contradiction.
\end{example}

Note that in the above example, the minimal number of generators of $c(g)$ localized at the homogeneous maximal ideal is infinite.  Hence, there is no natural choice of exponent, so it stands to reason that a counterexample would come from such a situation.  Accordingly, we ask the following question.

\vspace{3pt}

\noindent \textbf{Question:} Let $R$ be a commutative ring and $f, g\in R[\![X]\!]$.  Let $k$ be the supremum of the numbers $\mu(c(g)_\m)$, where $\m$ is taken over all maximal ideals of $R$, and \emph{suppose that $k<\infty$}.  Then does Equation~\ref{eqn:ceq} hold?

\vspace{3pt}

On the other hand, said formula does hold for \emph{some} non-Noetherian rings, even in some cases where the contents are not locally finitely generated.  Indeed, Anderson and Kang \cite[Theorem 2.4 and Corollary 2.6]{AnKa-content} have shown that if $R$ is a valuation ring whose value group $G$ is a subgroup of the real numbers (or more generally, a Pr\"ufer domain of dimension at most 1), then for any $f, g\in R[\![X]\!]$, one has $c(f)c(g)=c(fg)$, which is Equation~\ref{eqn:ceq} with $k=1$.  When $G \ncong \Z$ or $0$, it is easy over any such valuation ring to create power series with infinitely generated contents.

\begin{example}
Recall from the introduction that Rush had claimed (where for convenience, we switch the roles of $f$ and $g$ from that in his article) that if $R=k[u,v]$ ($k$ a field), $f=v+X$, and $g= u+vX(\sum_{i=0}^\infty X^i)$, then $c(f) = R$ but $c(fg) \neq c(g)$, and hence that Theorem~\ref{mainthm} cannot hold for this choice of $R$ \cite[p. 331]{Ru-content}.  Namely, he notes that $c(fg) = (uv, u+v^2, v+v^2)$ and $c(g) = (u,v)$, and then he claims that $c(fg) \neq c(g)$.  But in fact, $c(fg) = c(g)$, as can be seen by the equations \[
v=(-1)\cdot uv + v \cdot (u+v^2) + (1-v) \cdot (v+v^2)
\]
and \[
u-v = (u+v^2) - (v+v^2).
\]
\end{example}

\providecommand{\bysame}{\leavevmode\hbox to3em{\hrulefill}\thinspace}
\providecommand{\MR}{\relax\ifhmode\unskip\space\fi MR }
\providecommand{\MRhref}[2]{%
  \href{http://www.ams.org/mathscinet-getitem?mr=#1}{#2}
}
\providecommand{\href}[2]{#2}

\end{document}